\documentclass[11pt]{amsart}
\usepackage[english]{babel}
\usepackage{amsfonts,latexsym,amsthm,amssymb,graphicx}
\usepackage{mathabx}
\usepackage[all]{xy}
\usepackage[usenames]{color}
\usepackage{tikz}
\usetikzlibrary{shapes}
\usetikzlibrary{decorations.pathreplacing}
\usepackage{amsmath}
\usepackage{graphicx}
\usepackage[enableskew]{youngtab}
\usepackage[utf8]{inputenc}
\usepackage[english]{babel}
\usepackage{tikz}
\usetikzlibrary{arrows}
\usepackage{float}
\usepackage{amscd}
\usepackage{pstricks}
\usepackage{tableau}
\usepackage{multicol}
\usepackage{blindtext}
\usepackage{geometry}
\usepackage[alphabetic]{amsrefs}
\usepackage{enumerate}
\usepackage[colorlinks=true, linkcolor=blue, citecolor=blue, urlcolor=blue]{hyperref}
\usepackage{dynkin-diagrams}
\usetikzlibrary{decorations.pathreplacing}

\usepackage{color,soul}
\usepackage{ytableau}

\usepackage{graphicx}

\newcommand{\C}{{\mathbb C}}
\newcommand{\Z}{{\mathbb Z}}

\newcommand{\IG}{\mathrm{IG}}

\DeclareMathOperator{\Sp}{Sp}
\DeclareMathOperator{\GL}{GL}
\newcommand{\IF}{\mathrm{IF}}

\newcommand{\Xo}{X^\circ} 
\newcommand{\Tsmall}{T_{2n}}
\newcommand{\Wodd}{W^{odd}}
\newcommand{\Xev}{X^{ev}}

\newcommand{\Wlarge}{W}

\newcommand{\wtE}{\widetilde{E}}
\newcommand{\wto}{\widetilde{\omega}}
\DeclareMathOperator{\Span}{Span}
\DeclareMathOperator{\e}{\mathbf{e}}
\DeclareMathOperator{\diag}{diag}

\newcommand{\Mb}{\overline{\mathcal{M}}}

\DeclareMathOperator{\ev}{ev}

\DeclareFontFamily{OT1}{rsfs}{}
\DeclareFontShape{OT1}{rsfs}{n}{it}{<-> rsfs10}{}
\DeclareMathAlphabet{\mathscr}{OT1}{rsfs}{n}{it}

\CompileMatrices

\newtheorem{thm}{Theorem}[section]
\newtheorem{lemma}[thm]{Lemma}

\newtheorem{prop}[thm]{Proposition}

{\theoremstyle{definition} \newtheorem{defn}[thm]{Definition}}
{\theoremstyle{remark} \newtheorem{remark}[thm]{Remark}
\newtheorem{example}[thm]{Example}}

\begin{document}

\title[]{Curve neighborhoods and Combinatorial Property $\mathcal{O}$ for a family of odd symplectic partial flag manifolds}



\author{Connor Bean}

\address{
Department of Mathematical Sciences,
Henson Science Hall, 
Salisbury University,
Salisbury, MD 21801
}
\email{cbean2@gulls.salisbury.edu}

\author{Bradley Cruikshank}

\address{
Department of Mathematical Sciences,
Henson Science Hall, 
Salisbury University,
Salisbury, MD 21801
}
\email{bcruikshank1@gulls.salisbury.edu}

\author{Ryan M. Shifler}

\address{
Department of Mathematical Sciences,
Henson Science Hall, 
Salisbury University,
Salisbury, MD 21801
}
\email{rmshifler@salisbury.edu}


\begin{abstract}
Let $E$ be an odd dimensional complex vector space and $\IF:=\IF(1,2;E)$ be the family of odd symplectic partial flag manifold. In this paper we give a full description of the irreducible components of the degree $d$ curve neighborhood of any Schubert variety of $\IF$, study their lattice structure, and prove a combinatorial version of Conjecture $\mathcal{O}.$
\end{abstract}

\maketitle


%
%

\section{Introduction}
\subsection{Overview}
The degree $d$ curve neighborhood of a subvariety $V \subset X$, denoted $\Gamma_d(V)$, is the closure of the union of all degree $d$ rational curves through $V$. Curve neighborhoods were introduced in \cite{BCMP:qkfin} to prove finiteness of quantum $K$-theory for $X$ a cominuscule homogeneous space. Let $E$ be an odd dimensional complex vector space and $\IF:=\IF(1,2;E)$ be the family of odd symplectic partial flag manifold made up of sequences of vector spaces $(V_1 \subset V_2 \subset E)$ where $\dim V_i=i$ and $V_i$ is isotropic with respect to a (necessarily degenerate) symmetric form. In this paper we give a full description of the irreducible components of the degree $d$ curve neighborhood of any Schubert variety of $\IF$, study their lattice structure, and prove a combinatorial version of Conjecture $\mathcal{O}.$

\subsection{Broader context} Curve neighborhoods in homogeneous $G/P$ case are irreducible and there are combinatorial models to perform calculations (see \cite{BCMP:qkfin,buch.m:nbhds,ShiflerWithrow, Maya}). In \cite{Songul:thesis}, Aslan shows that the irreducible components of curve neighborhoods in the Affine Flag in Type A have equal dimension. It is shown in \cite{PechShif:CNBDS} that curve neighborhoods in the odd symplectic Grassmannian may not be irreducible. In Section \ref{sec:combinatorial}, we see, for the first time, that reducible curves are an integral part of understanding the geometry of the (combinatorial) quantum Bruhat graph for $\IF$. See Remark \ref{redcrv}.

\subsection{Curve neighborhoods}We now discuss curve neighborhoods. Let $X$ be a Fano variety (one could consider $X$ to be smooth). Let $d \in H_2(X,\Z)$ be an effective degree. Recall that the moduli space of genus $0$, degree $d$ stable maps with two marked points $\Mb_{0,2}(X,d)$ is endowed with two evaluation maps $\ev_i \colon \Mb_{0,2}(X,d) \to X$, $i=1,2$ which evaluate stable maps at the $i$-th marked point. Let $\Omega \subset X$ be a closed subvariety. The \emph{curve neighborhood} of $\Omega$ is the subscheme 
\[ 
    \Gamma_d(\Omega) := \ev_2( \ev_1^{-1} \Omega) \subset X
\] 
endowed with the reduced scheme structure.

This notion was introduced by Buch, Chaput, Mihalcea and Perrin \cite{BCMP:qkfin} to help study the quantum K-theory ring of cominuscule Grassmannians. It was analyzed further in \cite{buch.m:nbhds,PechShif:CNBDS}. Often, estimates for the dimension of the curve neighborhoods provide vanishing conditions for certain Gromov-Witten invariants.

\subsection{Lattices}
We will study curve neighborhoods through the lens of lattices. That is, let $X$ be a smooth variety containing the subvariety $\Omega$. It's interesting to ask if the set $\{ \Gamma_{d}(\Omega))\}_{d}$ forms a (distributive) lattice where $\leq$ is defined by inclusion of varieties. We will show that this set forms a distributive lattice when $X=\IF$ and $\Omega$ is a Schubert variety.

\subsection{Conjecture $\mathcal{O}$}
We will motivate and define the Combinatorial Property $\mathcal{O}$ which is known to correspond with Conjecture $\mathcal{O}$ in the homogeneous $G/P$ and odd symplectic Grassmannian cases. We begin by reviewing Conjecture $\mathcal{O}$ and conclude with its relation to graph theory in Lemma \ref{lemma: propO}. In Subsection \ref{combversion} we state the combinatorial versions of the quantum Bruhat graph and the Conjecture $\mathcal{O}$. In Subsection \ref{CCO} we state and prove that Combinatorial Property $\mathcal{O}$ holds for $\IF$ in Theorem \ref{Omainresult}.

We recall the precise statement of Conjecture $\mathcal{O}$. Let $X$ be a Fano variety, let $K:=K_X$ be the canonical line bundle of $X$, let $X_D$ be a fundamental divisor of $X$, and let $c_1(X):=c_1(-K) \in H^2(X)$ be the anticanonical class. The Fano index of $X$ is $r$, where $r$ is the greatest integer such that $K_X \cong -rX_D$. The small quantum cohomology ring $(QH^*(X), \star)$ is a graded algebra over $\mathbb{Z}[q]$, where $q$ is the quantum parameter. Consider the specialization $H^{\bullet}(X):=QH^*(X)|_{q=1}$ at $q=1$. The quantum multiplication by the first Chern class $c_1(X)$ induces an endomorphism $\hat{c}_1$ of the finite-dimensional vector space $H^{\bullet}(X)$: \[y \in H^{\bullet}(X) \mapsto \hat{c}_1(y):=(c_1(X)\star y)|_{q=1}. \]

Denote by $\delta_0:=\max \{|\delta| : \delta \mbox{ is an eigenvalue of } \hat{c}_1 \}.$ Then Property $\mathcal{O}$ states the following:
\begin{enumerate}
\item The real number $\delta_0$ is an eigenvalue of $\hat{c}_1$ of multiplicity one.
\item If $\delta$ is any eigenvalue of $\hat{c}_1$ with $|\delta|=\delta_0$, then $\delta=\delta_0 \gamma$ for some $r$-th root of unity $\gamma \in \mathbb{C}$, where $r$ is the Fano index of $X$.
\end{enumerate}
The property $\mathcal{O}$ was conjectured to hold for any Fano, complex manifold $X$ in \cite{GGI}. If a Fano, complex, manifold has Property $\mathcal{O}$ then we say that the space satisfies Conjecture $\mathcal{O}$. Conjecture $\mathcal{O}$ underlies Gamma Conjectures I and II of Galkin, Golyshev, and Iritani. The Gamma Conjectures refine earlier conjectures by Dubrovin on Frobenius manifolds and mirror symmetry. Conjecture $\mathcal{O}$ has already been proved for many cases in \cite{ChLi,LMS,delpezzo,projinter,BFSS}. The Perron-Frobenius theory of nonnegative matrices reduces the proofs that Conjecture $\mathcal{O}$ holds for the homogeneous and the odd symplectic Grassmannian cases to be a graph-theoretic check. This is because Conjecture $\mathcal{O}$ is largely reminiscent of Perron-Frobenius Theory.

The small quantum cohomology is defined as follows. Let $(\alpha_i)_i$ be a basis of $H^*(X)$, the classical cohomology ring, and let $(\alpha_i^{\vee})_i$ be the dual basis for the Poincar\'e pairing. The multiplication is given by \[ \alpha_i \star \alpha_j= \sum_{d\geq 0, k} c_{i,j}^{k,d}q^d\alpha_k \] where $c_{i,j}^{k,d}$ are the 3-point, genus 0, Gromov-Witten invariants corresponding to the classes $\alpha_i, \alpha_j$, and $\alpha_k^{\vee}$. We will make use of the quantum Chevalley formula which is the multiplication of a hyperplane class $h$ with another class $\alpha_j$. The explicit quantum Chevalley formula is the key ingredient used to prove Property $\mathcal{O}$ holds.

\subsection{Sufficient Criterion for Property $\mathcal{O}$ to hold} We recall the notion of the (oriented) quantum Bruhat graph of a Fano variety $X$. The vertices of this graph are the basis elements $\alpha_i \in H^{\bullet}(X):=QH^*(X)|_{q=1}$. There is an oriented edge $\alpha_i \rightarrow \alpha_j$ if the class $\alpha_j$ appears with positive coefficient (where we consider $q>0$) in the quantum Chevalley multiplication $h \star \alpha_i$ for some hyperplane class $h$. We say the graph is strongly connected if there are directed paths for $\alpha_i$ to $\alpha_j$ and $\alpha_j$ to $\alpha_i$ for any $\alpha_i,\alpha_j \in H^{\bullet}(X)$. If there is a path \[ \alpha_{i_1} \rightarrow \alpha_{i_2} \rightarrow \cdots \rightarrow \alpha_{i_k} \rightarrow \alpha_{i_1}\] then we say this is a cycle with cycle length $k$. Using the Perron-Frobenius theory of non-negative matrices, Conjecture $\mathcal{O}$ reduces to a graph-theoretic check of the quantum Bruhat graph. The techniques involving Perron-Frobenius theory used by Li, Mihalcea, and Shifler in \cite{LMS} and Cheong and Li in \cite{ChLi} imply the following lemma:

\begin{lemma}\label{lemma: propO}
If the following conditions hold for a Fano variety $X$:
\begin{enumerate}
\item The matrix representation of $\hat{c}_1$ is nonnegative,
\item The quantum Bruhat graph of $X$ is strongly connected, and 
\item There exists a set of cycles where the Great Common Divisor of the cycle lengths is $r$, the Fano index, in the quantum Bruhat graph of $X$,
\end{enumerate}
then Property $\mathcal{O}$ holds for $X$.
\end{lemma}

\begin{remark}
Part (2) of Conjecture $\mathcal{O}$ holds automatically if the Fano index is equal to one. That is, we don't need to verify Lemma 1.1 (3) for $\IF(1,2;E)$.
\end{remark}

\subsection{Combinatorial version of the quantum Bruhat graph and Conjecture $\mathcal{O}$} \label{combversion}

Let $X$ be a Fano variety. Let $\mathcal{B}:=\{\alpha_i\}_{i \in I}$ denote a basis of the cohomology ring $H^*(X)$. Denote its first Chern class by \[ c_1(X)=a_1 \mbox{Div}_1+a_2 \mbox{Div}_2+\cdots+a_k \mbox{Div}_k\] where $\mbox{Div}_i \in \mathcal{B}$ is a divisor class for each $1 \leq i \leq k$.

\begin{defn} The combinatorial quantum Bruhat graph of $X$ is define as follows. The vertices of this graph are the basis elements $\alpha_i \in H^*(X)$. The edges set is given as follows:
\begin{enumerate}
\item There is an oriented edge $\alpha_i \rightarrow \alpha_j$ if the class $\alpha_j$ appears with positive coefficient in the Chevalley multiplication $h \star \alpha_i$ for some hyperplane class $h$.
\item Let $\alpha_i=[X(i)]$ and $\alpha_j=[X(j)]$. There is an oriented edge $\alpha_i \rightarrow \alpha_j$ 
\begin{enumerate}
\item if $X(j)$ is an irreducible component of $\Gamma_d(X(i))$ where $d=(d_1,\cdots,d_k)$,
\item and \[ \dim(X(j))-\dim(X(i))=a_1d_1+a_2d_2+\cdots+a_kd_k-1.\]
\end{enumerate}
\end{enumerate}
\end{defn}

Lemma \ref{lemma: propO} leads us to naturally consider the following combinatorial formulation of Conjecture $\mathcal{O}$.

\begin{defn}
    Combinatorial Property $\mathcal{O}$ holds if the following conditions are satisfied:
    \begin{enumerate}
    \item There is a basis with respect to which the associated matrix of $\hat{c}_1$ is nonnegative\footnote{General Fano manifolds may not satisfy this property. For instance, the blowup of $\mathbb{P}^2$ at a point does not satisfy this property.},
    \item The combinatorial quantum Bruhat graph is strongly connected and the Greatest Common Divisor of the cycle lengths is $r=\mbox{GCD}(a_1,a_2,\cdots,a_k).$
    \end{enumerate}
\end{defn}

The purpose of Proposition \ref{QBG} is to assert that the Combinatorial versions of the quantum Bruhat graph and Conjecture $\mathcal{O}$ correspond with their geometric analogue.

\begin{prop}[\cite{LMS,ChLi}] \label{QBG}
The combinatorial quantum Bruhat graphs (Combinatorial Property $\mathcal{O}$) and the quantum Bruhat graphs (Conjecture $\mathcal{O}$) correspond for homogeneous $G/P$ and the odd symplectic Grassmannian.
\end{prop}

\subsection{Organization of the main results}
The manuscript has three objectives. The first is to compute the curve neighborhoods of Schubert varieties in $\IF$. This is done in Section \ref{sec:crvnbd}. We study curve neighborhoods in the context of lattices in Section \ref{sec:lattices}.  In Section \ref{sec:combinatorial} we study the combinatorial analogues of the quantum Bruhat graphs and Conjecture $\mathcal{O}$. We conclude by showing that Combinatorial Property $\mathcal{O}$ holds which motivate futher study of $\IF$.

\section{Notations and Definitions}
\subsection{A small family of odd symplectic partial flag manifolds} Let $n \geq 2$ and $E:=\mathbb{C}^{2n+1}$ be an odd-dimensional complex vector space. An {\it odd symplectic form} $\omega$ on $E$ is a skew-symmetric bilinear form of maximal rank (i.e. with kernel of dimension 1). It will be convenient to extend the form $\omega$ to a (nondegenerate) symplectic from $\tilde{\omega}$ on an even-dimensional space $\wtE \supset E$, and to identify $E \subset \wtE$ with a coordinate hyperplane $\mathbb{C}^{2n+1} \subset \mathbb{C}^{2n+2}$.

For that, let $\{ \e_1,\ldots , \e_{n+1}, \e_{\overline{n+1}}, \ldots , \e_{\bar{2}}, \e_{\bar{1}} \}$ be the standard basis of $\wtE:=\C^{2n+2}$, where $\bar{i}=2n+3-i$. Consider $\wto$ to be the symplectic form on $\wtE$ defined by 
\[ 
    \wto(\e_i,\e_j)=\delta_{i,\bar{j}} \text{ for all $1 \leq i\leq j \leq \bar{1}$}.
\]
The form $\wto$ restricts to the degenerate skew-symmetric form $\omega$ on \[ E=\C^{2n+1}=\left<\e_1, \e_2,\cdots, \e_{2n+1} \right>\] such that the kernel $\ker \omega$ is generated by $\e_1$. Then 
\[
    \omega(\e_i,\e_j)=\delta_{i,\bar{j}} \text{ for all  $1 \leq i\leq j \leq \bar{2}$}.
\] 
Let $F \subset E$ denote the $2n$-dimensional vector space with basis $\{ \e_2, \e_3,\cdots,\e_{2n+1} \}$.

The odd symplectic partial flag manifold we are considering is \[\IF(1,2; E):=\{V_1 \subset V_2 \subset E : \dim V_i=i \mbox{ and } \omega(x,y)=0 \mbox{ for any } x,y \in V_i\}.\] 

The restriction of $\omega$ to $F$ is non-degenerate, hence we can see the odd symplectic Grassmannian as intermediate space
\begin{equation}\label{E:evenodd} 
    \IF(1;F) \subset \IF(1,2; E) \subset \IF(1,2;\wtE), 
\end{equation} 
sandwiched between two odd symplectic flag manifolds. This and the more general ``odd symplectic partial flag varieties'' have been studied in \cite{mihai:odd,pech:quantum,mihalcea.shifler:qhodd,LMS,PechShif:CNBDS}. In particular, Mihai showed that $\IF$ is a smooth Schubert variety in $\IF(1,2;\wtE)$. 

\subsection{The odd symplectic group}

Proctor's \emph{odd symplectic group} (see \cite{proctor:oddsymgrps}) is the subgroup of $\GL(E)$ which preserves the odd symplectic form $\omega$: 
\[ 
    \Sp_{2n+1}(E) := \{ g \in \GL(E) \mid \omega(g \cdot u, g \cdot v) = \omega(u,v), \forall u,v \in E \}. 
\] 

Let $\Sp_{2n}(F)$ and $\Sp_{2n+2}(\widetilde{E})$ denote the symplectic groups which respectively preserve the symplectic forms $\omega_{\mid F}$ and $\wto$. Then with respect to the decomposition $E = F \oplus \ker \omega$ the elements of the odd symplectic group $\Sp_{2n+1}(E)$ are matrices of the form 
\begin{equation*}
    \Sp_{2n+1}(E) = \left\{  \begin{pmatrix}
        \lambda & a  \\
        0 & S  \\
    \end{pmatrix} \mid \lambda \in \C^*, a \in \C^{2n}, S \in \Sp_{2n}(F) \right\}.
\end{equation*}

The symplectic group $\Sp_{2n}(F)$ embeds naturally into $\Sp_{2n+1}(E)$ by $\lambda = 1$ and $a = 0$, but $\Sp_{2n+1}(E)$ is \emph{not} a subgroup of $\Sp_{2n+2}(\wtE)$.\begin{footnote}{However, Gelfand and Zelevinsky \cite{gelfand.zelevinsky} defined another group $\widetilde{\Sp}_{2n+1}$ closely related to $\Sp_{2n+1}$ such that $\Sp_{2n} \subset \widetilde{\Sp}_{2n+1} \subset \Sp_{2n+2}$.}\end{footnote}~Mihai showed in \cite{mihai:odd}*{Prop 3.3} that there is a surjection $P \to \Sp_{2n+1}(E)$ where $P \subset \Sp_{2n+2}(\wtE)$ is the parabolic subgroup which preserves $\ker \omega$, and the map is given by restricting $g \mapsto g_{|E}$. Then the Borel subgroup $B_{2n+2} \subset \Sp_{2n+2}(\wtE)$ of upper triangular matrices restricts to the (Borel) subgroup $B \subset \Sp_{2n+1}(E)$. Similarly, the maximal torus
\[
    T_{2n+2} := \{ \diag(t_1,\cdots,t_{n+1},t_{n+1}^{-1},\cdots, t_1^{-1}): t_1,\cdots,t_{n+1} \in \C^* \} \subset B_{2n+2}
\]
restricts to the maximal torus 
\[ 
    T=\{ \diag(t_1,\cdots,t_{n+1},t_{n+1}^{-1},\cdots, t_{2}^{-1}): t_1,\cdots,t_{n+1} \in \C^* \} \subset B.
\]
Later on we will also require notation for subgroups of $\Sp_{2n}(F)$, viewed as a subgroup of $\Sp_{2n+1}(E)$. We denote by $B_{2n} \subset B$ the Borel subgroup of upper-triangular matrices in $\Sp_{2n}(F)$ and by $\Tsmall$ the maximal torus
\[
    \Tsmall= \{ \diag(1,t_2,\cdots,t_{n+1},t_{n+1}^{-1},\cdots, t_{2}^{-1}): t_2,\cdots,t_{n+1} \in \C^* \} \subset B_{2n}.
\]

Mihai showed that the odd symplectic group $\Sp_{2n+1}(E)$ acts on $\IF$ with three orbits:
\begin{align*}
 \Xo &= \{ V \in \IF \mid \e_1 \notin V_2\} \text{ the open orbit}. \\
Z_2 &= \{ V \in \IF \mid \e_1 \in V_2 \backslash V_1\}.\\
Z_1 &= \{ V \in \IF \mid \e_1 \in V_1\} \text{ the closed orbit}.
\end{align*}
The closed orbit $Z_1$ is isomorphic to $\IF(1,F)$ via the map $V \mapsto V \cap F$.

\subsection{The Weyl group of \texorpdfstring{$\Sp_{2n+2}$}{the symplectic group} and odd symplectic minimal representatives}\label{s:weyl}

There are many possible ways to index the Schubert varieties of isotropic flag manifolds. Here we recall an indexation using
signed permutations. 

Consider the root system of type $C_{n+1}$ with positive roots
\[
    R^+ = \{ t_i \pm t_j \mid 1 \leq  i < j \leq n+1\} \cup \{ 2 t_i \mid 1 \leq  i \leq n+1 \}
\]
and the subset of simple roots 
\[
    \Delta = \{ \alpha_i := t_i-t_{i+1} \mid 1 \leq i \leq n \} \cup \{ \alpha_{n+1}:= 2t_{n+1} \}.
\]
The associated Weyl group $\Wlarge$ is the hyperoctahedral group consisting of \emph{signed permutations}, i.e. permutations $w$ of the elements $\{1, \cdots, n+1,\overline{n+1},\cdots,\overline{1}\}$ satisfying $w(\overline{i})=\overline{w(i)}$ for all $w \in W$. For $1 \leq i \leq n$ denote by $s_i$ the simple reflection corresponding to the root $t_i - t_{i+1}$ and $s_{n+1}$ the simple reflection of $2 t_{n+1}$. In particular, if $1 \leq i \leq n$ then $s_i(i)=i+1$, $s_i(i+1)=i$, and $s_i(j)$ is fixed for all other $j$. Also, $s_{n+1}(n+1)=\overline{n+1}$, $s_{n+1}(\overline{n+1})=n+1$, and $s_{n+1}(j)$ is fixed for all other $j$.

Each subset $ I:=\{ i_1 < \ldots < i_r \} \subset \{ 1, \ldots , n+1 \}$ determines a parabolic subgroup $P:=P_I \leq \Sp_{2n+2}(\wtE)$ with Weyl group $W_{P} = \langle s_i \mid i \neq i_j \rangle$ generated by reflections with indices \emph{not} in $I$. Let $\Delta_P:= \{ \alpha_{i_s} \mid i_s \notin \{ i_1, \ldots , i_r \} \}$ and $R_P^+ := \Span_\Z \Delta_P \cap R^+$; these are the positive roots of $P$. Let $\ell \colon W \to \mathbb{N}$ be the length function and denote by $W^{P}$ the set of minimal length representatives of the cosets in $W/W_{P}$. The length function descends to $W/W_P$ by $\ell(u W_P) = \ell(u')$ where $u' \in W^P$ is the minimal length representative for the coset $u W_P$. We have a natural ordering 
\[ 
    1 < 2 < \cdots < n+1 < \overline{n+1} < \cdots < \overline{1},
\]
which is consistent with our earlier notation $\overline{i} := 2n+3 - i$. 

Let $P$ be the parabolic obtained by excluding the reflections $s_1$ and $s_2$. Then the minimal length representatives $W^P$ have the form $(w(1)|w(2)|w(3)<\cdots <w(n) \leq n+1)$. Since the last $n-1$ labels are determined from the first 2 labels, we will identify an element in $W^P$ with $(w(1)|w(2))$.


\begin{example}\label{Ex:minlengthrep}
The reflection $s_{t_1+t_2}$ is given by the signed permutation \[s_{t_1+t_2}(1)=\bar{2},  s_{t_1+t_2}(2)=\bar{1}, \mbox{ and }s_{t_1+t_2}(i)=i \mbox{ for all }3 \leq i \leq n+1.\] The minimal length representative of $s_{t_1+t_2}\Wlarge^P$ is $(\bar{2}|\bar{1})$.
\end{example}


\subsection{Schubert Varieties in even and odd symplectic partial flag manifolds} 

Recall that the even symplectic partial flag manifold $\Xev=\IF(1,2; \tilde{E})$ is a homogeneous space $\Sp_{2n+2}/P$, where $P$ is the parabolic subgroup generated by the simple reflections $s_i$ with $i \neq 1,2$. For each $w \in \Wlarge^P$ let $\Xev(w)^\circ:=B_{2n+2} w B_{2n+2}/P$ be the \emph{Schubert cell}. This is isomorphic to the space $\C^{\ell(w)}$. Its closure $\Xev(w):=\overline{\Xev(w)^\circ}$ is the \emph{Schubert variety}. We might occasionally use the notation $\Xev(w \Wlarge_P)$ if we want to emphasize the corresponding coset, or if $w$ is not necessarily a minimal length representative. Recall that the Bruhat ordering can be equivalently described by $v \leq w$ if and only if $\Xev(v) \subset \Xev(w)$. Set 
\begin{equation*}\label{E:w0} 
    w_0 = (\bar{2}|\bar{3})
\end{equation*}
this is an element in $\Wlarge$.
Recall that the odd symplectic Borel subgroup is $B= B_{2n+2} \cap \Sp_{2n+1}$. The following results were proved by Mihai \cite{mihai:odd}*{\S 4}. 

\begin{remark}
Here $w_0$ is the longest element for $\IF$ which is different than the longest element for $\IF(1,2;\tilde{E})$.
\end{remark}
\begin{prop}\label{prop:schubert}
\begin{enumerate}[(a)]
    \item The natural embedding $\iota: X=\IF \hookrightarrow \Xev=\IF(1,2;\tilde{E})$ identifies $\IF$ with the (smooth) Schubert subvariety 
    \[ 
        \Xev(w_0 \Wlarge_P) \subset \IF(1,2;\tilde{E}).
    \] 
    \item The Schubert cells (i.e. the $B_{2n+2}$-orbits) in $\Xev(w_0)$ coincide with the $B$-orbits in $\IF$. In particular, the $B$-orbits in $\IF$ are given by the Schubert cells $\Xev(w)^\circ \subset \IF(1,2;\tilde{E})$ such that $w \leq w_0$.
\end{enumerate}
\end{prop} 

We discuss Schubert cells or varieties in the odd symplectic case. For each $w \leq w_0$ such that $w \in \Wlarge^P$, we denote by $X(w)^\circ$, and $X(w)$, the Schubert cell, respectively the Schubert variety in $\IF$. The same Schubert variety $X(w)$, but regarded in the even symplectic partial flag manifold is denoted by $\Xev(w)$. For further use we note that $\IF$ has complex dimension $\ell(\bar{2}|\bar{3})=4n-6$, $\IF(1,2;\tilde{E})$ has complex dimension $\ell(\bar{1}|\bar{2})=4n-4$, and $\IF$ has codimension 2 in $\IF(1,2;\tilde{E})$. Further, a Schubert variety $X(w)$ in $\IF$ is included in the closed $\Sp_{2n+1}$-orbit $Z_1$ of if and only if it has a minimal length representative $w \leq w_0$ such that $w(1)=1$.

Define the set $\Wodd:= \{ w \in \Wlarge \mid w \leq w_0 \}$ and call its elements \emph{odd symplectic permutations}. The set $\Wodd$ consists of permutations $w \in W$ such that $w(j) \neq \bar{1}$ for any $1 \leq j \leq n+1$ \cite{mihai:odd}*{Prop. 4.16}.

\subsection{Divisors and the first Chern class}

\begin{lemma}
The odd symplective flag manifold $\IF$ has two divisor classes. If $n=2$ then 
\[[X(Div_1)]:=[X(\bar{3} \mid \bar{2})] \hspace{12pt} and \hspace{12pt} [X(Div_2)]:=[X(\bar{2} \mid 3)].\]

If $n>2$ then
\[[X(Div_1)]:=[X(\bar{3} \mid \bar{2})] \hspace{12pt} and \hspace{12pt} [X(Div_2)]:=[X(\bar{2} \mid \bar{4})].\]

\end{lemma}

\begin{lemma}
The first Chern class of $\IF$ is
\[ c_1(\IF)=2 \cdot [X(Div_1)]+(2n-1) \cdot [X(Div_2)].\]
\end{lemma}

\section{The Moment Graph}
Sometimes called the GKM graph, the \emph{moment graph} of a variety with an action of a torus $T$ has a vertex for each $T$-fixed point, and an edge for each $1$-dimensional torus orbit. The description of the moment graphs for flag manifolds is well known, and it can be found e.g in \cite{kumar:kacmoody}*{Ch. XII}. In this section we consider the moment graphs for $X=\IF(1,2; E) \subset \Xev = \IF(1,2;\tilde{E})$. As before let $P \subset \Sp_{2n+2}$ be the maximal parabolic for $\Xev$. Recall that we will identify an element in $W^P$ with $(w(1)|w(2))$.

\subsection{Moment graph structure of $\IF(1,2;\tilde{E})$}
The moment graph of $\Xev$ has a vertex for each $w \in W^P$, and an edge $w\rightarrow ws_{\alpha}$ for each \[ 
    \alpha \in R^+ \setminus R_{P}^+ = \{ t_i - t_j \mid 1 \leq i \leq 2, i< j \leq n+1 \} \cup \{ t_i + t_j, 2 t_i \mid 1 \leq i \leq 2, 1 \leq i <j \leq n+1  \}. 
\]
Geometrically, this edge corresponds to the unique torus-stable curve $C_\alpha(w)$ joining $w$ and $ws_\alpha$. The curve $C_\alpha(w)$ has degree $d=(d_1,d_2)$, where $\alpha^\vee + \Delta_P^\vee = d_1 \alpha_1^\vee +d_2 \alpha_2^\vee +\Delta_P^\vee$. In the next section we classify the positive roots by their degree. In order to perform curve neighborhood calculations we will we give precise combinatorial description of the moment graph.

\begin{defn}\label{def:moment-graph-combinat}
Define the following to describe moment graph combinatorics.

\begin{enumerate}
\item Define the following four sets which partitions $R^+ \setminus R_{P}^+$.
\begin{enumerate}
\item $R^+_{(1,0)}=\{t_1-t_2\}$;
\item $R^+_{(0,1)}=\{t_2 \pm t_j \mid 3 \leq j \leq n+1 \} \cup \{2t_2\}$;
\item $R^+_{(1,1)}=\{t_1 \pm t_j: 3 \leq j \leq n+1 \} \cup \{2t_1\}$;
\item $R^+_{(1,2)}=\{t_1+t_2\}$.
\end{enumerate}
\item A {\it chain of degree $d$} is a path in the (unoriented) moment graph where the sum of edge degrees equals $d$. We will often use that notation $uW_P \overset{d}{\rightarrow} vW_P$ to denote such a path.
\end{enumerate}
\end{defn}
In the next lemma we give a formula for the degree $d$ of a chain which is useful to calculate curve neighborhoods. In particular, we will see that the degree of a chain is determined by summing the weights of the edges traversed in the moment graph.
\begin{lemma} \label{lem:3.2}
Let $u,v \in W^P$ be connected by a degree $d$ chain
\[
    (u W_P \overset{d} \to vW_P) = (uW_P \to us_{\alpha_1} W_P \to \dots \to us_{\alpha_1}s_{\alpha_2}\ldots s_{\alpha_t}W_P)
\]
where $vW_P=us_{\alpha_1}s_{\alpha_2}\ldots s_{\alpha_t}W_P$ and the $\alpha_i$ are in $ R^+ \backslash R^+_P$. Then:

\begin{eqnarray*}
d=\left(\#\left\{\alpha_i \in R^+_{(1,0)}\right\}, \#\left\{\alpha_i \in R^+_{(0,1)}\right\} \right)&+&\left(\#\left\{\alpha_i \in R^+_{(1,1)}\right\}, \#\left\{\alpha_i \in R^+_{(1,1)}\right\} \right)\\
&+&\left(\#\left\{\alpha_i \in R^+_{(1,2)}\right\}, 2 \cdot \#\left\{\alpha_i \in R^+_{(1,2)}\right\} \right).
\end{eqnarray*}
\end{lemma}

\subsection{Moment graph structure of $\IF$}

The moment graph of $\IF$ is the full subgraph of $\IF(1,2;2n+2)$ determined by the vertices $w \in W^P \cap  W^{odd}$. Notice that the orbits of $T$ and $T_{2n+2}$ coincide, therefore we do not distinguish between the moment graphs for these tori. See Figure \ref{IFn=2MG}.

\begin{figure}
\caption{ The moment graph for $\IF$ when $n=2$.}
\label{IFn=2MG}
\begin{tikzpicture}[scale=.65]
\tikzstyle{vertex} = [circle, minimum size = 2pt, inner sep = 0cm]
\tikzstyle{ledge} = [green]
\tikzstyle{redge}= [orange]
\tikzstyle{bedge} = [ blue]
\tikzstyle{twoedge} = [purple]

\node[vertex](v1) at (0,1) {$(1 \mid 2)$};
\node[vertex](v2) at (-2,3) {$(2 \mid 1)$};
\node[vertex](v3) at (2,3) {$(1 \mid 3)$};
\node[vertex](v4) at (-3,5) {$(2 \mid 3)$};
\node[vertex](v5) at (0,5) {$(3 \mid 1)$};
\node[vertex](v6) at (3,5) {$(1 \mid \bar{3})$};
\node[vertex](v7) at (-4.5,7) {$(3 \mid 2)$};
\node[vertex](v8) at (-2,7) {$(2 \mid \bar{3})$};
\node[vertex](v9) at (2,7) {$(\bar{3} \mid 1)$};
\node[vertex](v10) at (4.5,7) {$(1 \mid \bar{2})$};
\node[vertex](v11) at (-3,9) {$(\bar{3} \mid 2)$};
\node[vertex](v12) at (0,9) {$(3 \mid \bar{2})$};
\node[vertex](v13) at (3,9) {$(\bar{2} \mid 1)$};
\node[vertex](v14) at (-2,11) {$(\bar{2} \mid 3)$};
\node[vertex](v15) at (2,11) {$(\bar{3} \mid \bar{2})$};
\node[vertex](v16) at (0,13) {$(\bar{2}\mid \bar{3})$};

\draw[ledge](v1) to [bend right=10](v2);
\draw[redge](v1) to [bend right=5](v3);
\draw[redge](v1) to (v6);
\draw[bedge](v1) to [bend left=20](v7);
\draw[bedge](v1) to [bend right=45](v11);
\draw[redge](v1) to [bend right=20](v10);

\draw[redge](v2) to [bend right=5](v4);
\draw[redge](v2) to (v8);
\draw[bedge](v2) to [bend left=20](v5);
\draw[bedge](v2) to [bend left=20](v9);
\draw[bedge](v2) to [bend left=20](v13);

\draw[bedge](v3) to [bend left=20](v4);
\draw[bedge](v3) to [bend right=50](v14);
\draw[ledge](v3) to [bend right=10](v5);
\draw[redge](v3) to [bend right=10](v10);
\draw[redge](v3) to [bend right=5](v6);

\draw[ledge](v4) to [bend right=10](v7);
\draw[bedge](v4) to [bend left=35](v14);
\draw[twoedge](v4) to [bend right=20](v15);
\draw[redge](v4) to [bend right=5](v8);

\draw[bedge](v5) to [bend left=20](v13);
\draw[redge](v5) to [bend right=5](v7);
\draw[redge](v5) to (v12);
\draw[bedge](v5) to [bend left=20](v9);

\draw[bedge](v6) to [bend left=20](v8);
\draw[bedge](v6) to [bend right=50](v16);
\draw[ledge](v6) to [bend right=10](v9);
\draw[redge](v6) to [bend right=5](v10);

\draw[bedge](v7) to [bend left=20](v11);
\draw[twoedge](v7) to [bend right=20](v16);
\draw[redge](v7) to [bend right=5](v12);

\draw[ledge](v8) to [bend right=10](v11);
\draw[bedge](v8) to [bend left=20](v16);
\draw[twoedge](v8) to [bend right=20](v12);

\draw[bedge](v9) to [bend left=20](v13);
\draw[redge](v9) to [bend right=5](v11);
\draw[redge](v9) to (v15);

\draw[bedge](v10) to [bend left=20](v12);
\draw[bedge](v10) to [bend left=25](v15);
\draw[ledge](v10) to [bend left=5](v13);

\draw[twoedge](v11) to [bend right=20](v14);
\draw[redge](v11) to [bend right=5](v15);

\draw[ledge](v12) to [bend right=10](v14);
\draw[bedge](v12) to [bend left=20](v15);

\draw[redge](v13) to [bend right=5](v14);
\draw[redge](v13) to [bend right=30](v16);

\draw[redge](v14) to [bend left=5](v16);

\draw[ledge](v15) to [bend right=10](v16);

\node[vertex](v17) at (-2.5,0) {$x$};
\node[vertex](v18) at (2.5,0) {$y$};
\node[vertex](v19) at (-5.5,0) {$x$};
\node[vertex](v20) at (-3.5,0) {$y$};
\node[vertex](v21) at (3.5,0) {$x$};
\node[vertex](v22) at (5.5,0) {$y$};
\node[vertex](v26) at (-.5,0) {$y$};
\node[vertex](v27) at (.5,0) {$x$};

\node[vertex](v23) at (-4.5,-.5) {$(1,0)$};
\draw[ledge](v19)--(v20);

\node[vertex](v24) at (-1.5,-.5) {$(0,1)$};
\draw[redge](v17)--(v26);

\node[vertex](v28) at (1.5,-.5) {$(1,1)$};
\draw[bedge](v27)--(v18);

\node[vertex](v25) at (4.5,-.5) {$(1,2)$};
\draw[twoedge](v21)--(v22);

\end{tikzpicture}
\end{figure}
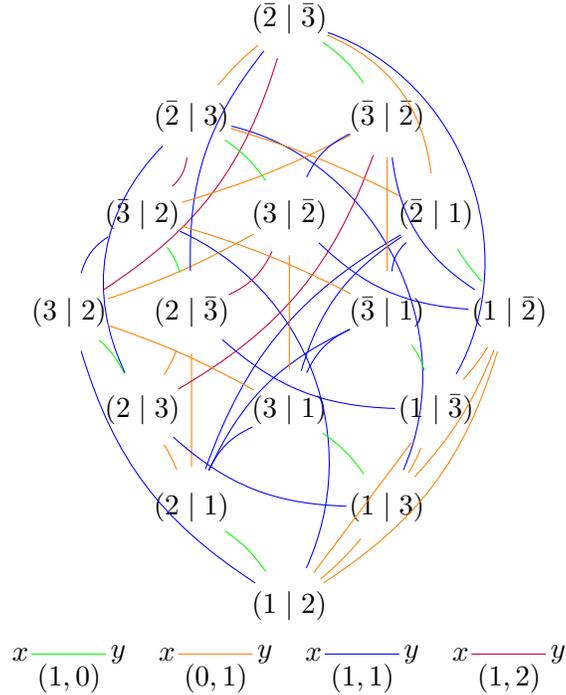

\section{Curve neighborhoods} \label{sec:crvnbd}
The main result of this section is Theorem \ref{main:crvnbd} which states all curve neighborhoods of Schubert varieties $\IF$. We will define a curve neighborhood in Definition \ref{def:defofcrvnbhd} and Proposition \ref{prop:moment-odd} will state the combinatorial equivalent version. Then Lemmas \ref{lem1} and \ref{lem2} are used to prove the main result in Theorem \ref{main:crvnbd}.

Let $X$ be a Fano variety. Let $d \in H_2(X,\Z)$ be an effective degree. Recall that the moduli space of genus $0$, degree $d$ stable maps with two marked points $\Mb_{0,2}(X,d)$ is endowed with two evaluation maps $\ev_i \colon \Mb_{0,2}(X,d) \to X$, $i=1,2$ which evaluate stable maps at the $i$-th marked point.

\begin{defn} \label{def:defofcrvnbhd}
Let $\Omega \subset X$ be a closed subvariety. The \emph{curve neighborhood} of $\Omega$ is the subscheme 
\[ 
    \Gamma_d(\Omega) := \ev_2( \ev_1^{-1} \Omega) \subset X
\] 
endowed with the reduced scheme structure.
\end{defn}

The next proposition gives a combinatorial formulation of curve neighborhoods in terms of the moment graph. See Figure \ref{crvnbdexample} for an example of an application of Proposition \ref{prop:moment-odd}.

\begin{prop}[\cite{buch.m:nbhds}]\label{prop:moment-odd}
Let $w \in W^P \cap \Wodd$. In the moment graph of $X=\IF$, let $\{v^1, \cdots, v^s\}$ be the maximal vertices (for the Bruhat order) which can be reached from any $u \leq w$ using a chain of degree $d$ or less. Then $\Gamma_d(X(w))=X(v^1) \cup \cdots \cup X(v^s)$.
\end{prop}

\begin{proof}
Let $Z_{w,d}=X(v^1) \cup \cdots \cup X(v^s)$. Let $v:= v^i \in Z_{w,d}$ be one of the maximal $T$-fixed points. By the definition of $v$ and the moment graph there exists a chain of $T$-stable rational curves of degree less than or equal to $d$ joining $u \leq w$ to $v$. It follows that there exists a degree $d$ stable map joining $u \leq w$ to $v$. Therefore $v \in \Gamma_d(X(w))$, thus $X(v) \subset \Gamma_d(X(w))$, and finally $Z_{w,d} \subset \Gamma_d(X(w))$.

For the converse inclusion, let $v \in \Gamma_d(X(w))$ be a $T$-fixed point. By \cite{mare.mihalcea}*{Lemma 5.3} there exists a $T$-stable curve joining a fixed point $u \in X(w)$ to $v$. This curve corresponds to a path of degree $d$ or less from some $u \leq w$ to $v$ in the moment graph of $\IG(k,2n+1)$. By maximality of the $v^i$ it follows that $v \leq v^i$ for some $i$, hence $v  \in X(v^i) \subset Z_{w,d}$, which completes the proof.
\end{proof}

\begin{figure}
\caption{In this figure we calculate a few curve neighborhoods of the Schubert point $(1|2)$ for $n=2$ as an example of Proposition \ref{prop:moment-odd}.}
\label{crvnbdexample}
\begin{center}
\begin{tikzpicture}[scale=.5]

\tikzstyle{vertex} = [circle, minimum size = 2pt, inner sep = 0pt]
\tikzstyle{ledge} = [ green]
\tikzstyle{redge}= [ orange]
\tikzstyle{bedge} = [ blue]
\tikzstyle{twoedge} = [ purple]

\node[vertex](v1) at (0,1) {$(1 \mid 2)$};
\node[vertex](v2) at (-2,3) {$(2 \mid 1)$};
\node[vertex](v10) at (4.5,7) {$(1 \mid \bar{2})$};
\node[vertex](v11) at (-3,9) {$(\bar{3} \mid 2)$};
\node[vertex](v13) at (3,9) {$(\bar{2} \mid 1)$};

\draw[ledge](v1) to [bend right=10](v2);
\draw[bedge](v1) to [bend right=45](v11);
\draw[redge](v1) to [bend right=20](v10);









\draw[ledge](v10) to [bend left=5](v13);






\node[vertex](v17) at (-2.5,0) {$x$};
\node[vertex](v18) at (2.5,0) {$y$};
\node[vertex](v19) at (-5.5,0) {$x$};
\node[vertex](v20) at (-3.5,0) {$y$};
\node[vertex](v21) at (3.5,0) {$x$};
\node[vertex](v22) at (5.5,0) {$y$};
\node[vertex](v26) at (-.5,0) {$y$};
\node[vertex](v27) at (.5,0) {$x$};

\node[vertex](v23) at (-4.5,-.5) {$(1,0)$};
\draw[ledge](v19)--(v20);

\node[vertex](v24) at (-1.5,-.5) {$(0,1)$};
\draw[redge](v17)--(v26);

\node[vertex](v28) at (1.5,-.5) {$(1,1)$};
\draw[bedge](v27)--(v18);

\node[vertex](v25) at (4.5,-.5) {$(1,2)$};
\draw[twoedge](v21)--(v22);


\end{tikzpicture}

$\Gamma_{(1,0)}(X(1|2))=X(2|1)$ \hspace{10pt} $\Gamma_{(0,1)}(X(1|2))=X(1|\bar{2})$ \hspace{10pt} $\Gamma_{(1,1)}(X(1|2))=X(\bar{3}|2) \cup X(\bar{2}|1)$
\end{center}
\end{figure}
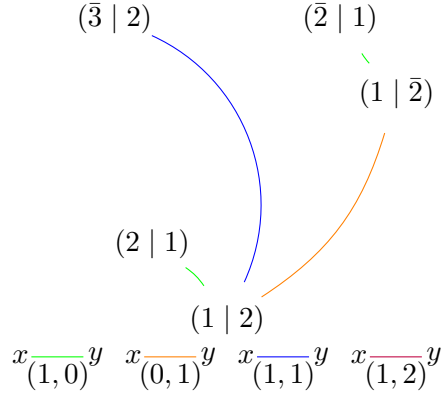

The next theorem is the main result of this section that states the precise calculations of curve neighborhoods of Schubert varieties in $\IF$. The result follows from Lemmas \ref{lem1} and \ref{lem2} stated below.

\begin{thm} \label{main:crvnbd}The following curve neighborhood calculations hold.
\begin{enumerate}
\item $\Gamma_{(d_1 \geq 1,0)}(X(a|b)) =
\begin{cases} 
    X(a|b);  a>b \\
    X(b|a);  a<b 
\end{cases}$
\vspace{12pt}
\item $\Gamma_{(0,d_2 \geq 1)}(X(a|b)) =
\begin{cases} 
    X(a|\bar{3});  a \in \{2,\bar{2} \} \\
    X(a|\bar{2});  a \notin \{2, \bar{2} \} 
\end{cases}$
\vspace{12pt}
\item $\Gamma_{(d_1 \geq 1,1)}(X(a|b)) = 
\begin{cases}
    X(\bar{3}|2) \cup X(\bar{2}|1); (a|b) \in \{(1|2),(2|1)\} \\
    X(\bar{2}|\max\{a,b\});  1 \leq a,b \leq \bar{3} \mbox{ and } (a|b)\notin \{(1|2),(2|1)\} \\
    X(\bar{2}|\bar{3});  \bar{2} \in \{a,b\}
\end{cases}$
\vspace{12pt}
\item $\Gamma_{(d_1 \geq 1,d_2 \geq 2)}(X(a|b))=X(\bar{2}|\bar{3})$
\end{enumerate}

\end{thm}

\begin{proof}
 For case (1), if $a>b$ then $(a|b) \cdot s_1=(a|b)$ and if $a<b$ then $(a|b) \cdot s_1=(b|a).$ 
 
 For case (2) we will need to check four subcases. Notice that we are multiplying $(a|b)$ by reflections of the form $t_2\pm t_j$ for $3 \leq j \leq n+1$ or $2t_2$. In particular, $a$ remains fixed so $\Gamma_{(0,1)}(X(a|b))=\Gamma_{(0,d_2\geq 1)}(X(a|b))$. If $a \in \{2, \bar{2} \}$ and $b \notin \{3,\bar{3} \}$ then $(a|b) \cdot s_{t_2+t_3}=(a|\bar{3})$. If $a \in \{2, \bar{2} \}$ and $b \in \{3,\bar{3} \}$ then $(a|b) \cdot s_{2t_2}=(a|\bar{3})$. If $a \notin \{2, \bar{2} \}$ and $b \notin \{2,\bar{2} \}$ then $(a|b) \cdot s_{t_2+t_3}=(a|\bar{2})$. If $a \notin \{2, \bar{2} \}$ and $b \in \{2,\bar{2} \}$ then $(a|b) \cdot s_{2t_2}=(a|\bar{2})$. In each case, this is the vertex with the greatest length that can be reached by and edge of degree $(0,1)$ from $(a|b)$. 
 
 For case (3), first notice that $\Gamma_{(1,1)}(X(a|b))=\Gamma_{(d_1\geq 1,1)}(X(a|b))$ since the reflection $s_1$ only interchanges $(a|b)$ and there is not an edge in the moment graph of $\IF$ with degree $(2,1)$. Case (3) follows by Lemma \ref{lem2}. 
 
 For case (4), notice that $\Gamma_{(1,2)}(X(a|b))=\Gamma_{(d_1\geq 1, d_2 \geq 1)}(X(a|b))$. By case (3), we see that $\Gamma_{(1, 1)}(X(a|b))$ has a component of the form $X(\bar{2}|b)$. The result for case (4) follows by an application of case (2). The result follows.
\end{proof}

In the next lemma we state the possible chains in the moment graph of $\IF$ of degree less than or equal to $(1,1)$.

\begin{lemma} \label{lem1}
Let $(a|b) \in W^{odd}$. Then we have the following chains in the moment graph of $\IF.$
\begin{enumerate}
\item For chains traversing an edge of degree $(1,0)$ followed by an edge of degree $(0,1)$, we have the following:
\begin{eqnarray*}
(a|b) \overset{(1,0)}{\longrightarrow} (b|a) \overset{(0,1)}{\longrightarrow} (b| h)
\end{eqnarray*}
where $h \notin \{b, \bar{b} \}$.

\item For chains traversing an edge of degree $(0,1)$ followed by an edge of degree $(1,0)$ we have the following:
\begin{eqnarray*}
(a|b) \overset{(0,1)}{\longrightarrow} (a|h) \overset{(1,0)}{\longrightarrow} (h| a)
\end{eqnarray*}
where $h \notin \{a, \bar{a} \}$.
Notice that if $a=\bar{2}$ then permutation length decreases on the second step.

\item For chains traverse an edge of degree $(1,1)$ we have the following:
\begin{eqnarray*}
(a|b) \overset{(1,1)}{\longrightarrow} (h|b)
\end{eqnarray*}
where $h \notin \{b, \bar{b} \}$.
\end{enumerate}
\end{lemma}

\begin{proof} We will prove each case. For case (1), following an edge of degree $(1,0)$ from $(a|b)$ results in $(b|a)$ by Lemma \ref{lem:3.2}. Then following an edge of degree $(0,1)$ from $(b|a)$ results in $(b|h)$ where $h \notin \{b, \bar{b} \}$ by Lemma \ref{lem:3.2}. 

For case (2), following an edge of degree $(0,1)$ from $(a|b)$ results in $(a|h)$ where $h \notin \{a, \bar{a}\}$ by Lemma \ref{lem:3.2}. Then following an edge of degree $(1,0)$ from $(a|h)$ results in $(h|a)$ by Lemma \ref{lem:3.2}. 

For case (3), following an edge of degree $(1,1)$ from $(a|b)$ results in $(h|b)$ where $h \notin \{b, \bar{b}\}$ by Lemma \ref{lem:3.2}. The result follows.
\end{proof}

In this lemma, we calculate the curve neighborhood $\Gamma_{(d_1 \geq 1,1)}(X(a|b))$ for any $(a|b) \in W^P$.

\begin{lemma} \label{lem2} First note the $\Gamma_{(d_1 \geq 1,1)}(X(a|b))=\Gamma_{(1,1)}(X(a|b))$ for any $(a|b) \in W^{odd}$ since right multiplication by $s_1$ on $w \in W$ interchanges $w(1)$ and $w(2)$. The curve neighborhood $\Gamma_{(1,1)}(X(a|b))$ is given by one of the following.
\begin{enumerate}
\item Suppose $a\notin \{2,\bar{2} \}$ and $b\notin \{2,\bar{2}\}$. 
\begin{enumerate}
\item If $a<b$ then

\[\Gamma_{(1,1)}(X(a|b))=X(\bar{2}|b). \]

\item If $a>b$ then

\[\Gamma_{(1,1)}(X(a|b))=X(\bar{2}|a). \]
\end{enumerate}

\item Suppose $a\notin \{2,\bar{2}\}$ and $b \in \{2,\bar{2}\}$.

\begin{enumerate}
\item If $b=\bar{2}$ then
\[\Gamma_{(1,1)}(X(a|b))=X(\bar{2}|\bar{3}). \]

\item If $a>2$ and $b=2$ then
\[\Gamma_{(1,1)}(X(a|b))=X(\bar{2}|a). \]

\item If $a=1$ and $b=2$ then

\[\Gamma_{(1,1)}(X(1|2))=X(\bar{3}|2)\cup X(\bar{2}|1). \]
\end{enumerate}

\item Suppose $a \in \{2,\bar{2}\}$ and $b \notin \{2,\bar{2}\}$.
\begin{enumerate}

\item If $a=\bar{2}$ then
\[\Gamma_{(1,1)}(X(a|b))=X(\bar{2}|\bar{3}). \]

\item If $a=2$ and $b>2$ then
\[\Gamma_{(1,1)}(X(a|b))=X(\bar{2}|b). \]

\item If $a=2$ and $b=1$ then
\[\Gamma_{(1,1)}(X(2|1))=X(\bar{3}|2)\cup X(\bar{2}|1). \]
\end{enumerate}

\end{enumerate}
\end{lemma}

\begin{proof}
The result follows from Lemma \ref{lem1} and the observation that the pair $\{ (\bar{3}|2)$, $(\bar{2}|1)\}$ is incomparable in the Bruhat order. The result follows.
\end{proof}

\section{Lattices} \label{sec:lattices}
Let $X$ be a Fano variety containing the subvariety $\Omega$. It is interesting to ask if the set $\{ \Gamma_{d}(\Omega))\}_{d}$ forms a (distributive) lattice where $\leq$ is defined by inclusion of varieties. We will show that  \[\{ \Gamma_{(d_1,d_2)}(X(a|b))\}_{(d_1,d_2) \geq (0,0)}\] is a distributive lattice for any $(a|b) \in W^{odd}$. We will review the definition of a (distributive) lattice and two particular lattices next.

\begin{defn}We will define lattices, distributive lattices, and two useful lattices.
\begin{enumerate}
\item A partially ordered set $(L, \leq)$ is a lattice if the following two conditions hold.
\begin{enumerate}
\item if for any $a,b \in L$ there is a unique element denoted by $a\vee b \in L$ such that 
\begin{itemize}
\item $a \leq a\vee b$ and $b \leq a\vee b$
\item and if there is a $c \in L$ such that $a \leq c$ and $b \leq c$ then $a \vee b \leq c$.
\end{itemize}
\item if for any $a,b \in L$ there is a unique element denoted by $a \wedge b \in L$ such that 
\begin{itemize}
\item $a \geq a\vee b$ and $b \geq a\vee b$
\item and if there is a $c \in L$ such that $a \geq c$ and $b \geq c$ then $a \wedge b \geq c$.
\end{itemize}
\end{enumerate}

\item A Lattice $(L,\leq)$ is a distributive lattice if \[ a \vee (b \wedge c)=(a \vee b) \wedge(a \vee c)\] for any $a,b,c \in L$.

\item Define the lattice $M_3$ and $N_5$ as follows:
\begin{multicols}{2}
\begin{enumerate}
\item $M_3=\{0,a,b,c,1 \}$ with $\leq$ defined by 
\begin{itemize}
\item $0 < a,b,c$, 
\item $a,b,c<1$, and 
\item $a,b,$ and $c$ are incomparable.
\end{itemize}

\columnbreak
\item $N_5=\{0,a,b,c,1\}$ with $\leq$ defined by
\begin{itemize}
    \item $0<c<a<1$
    \item $0<b<1$
    \item $b$ is incomparable with both $a$ and $c$.
\end{itemize}
\end{enumerate}
\end{multicols}
See Figure \ref{M3N5}.
\end{enumerate}
\end{defn}
    

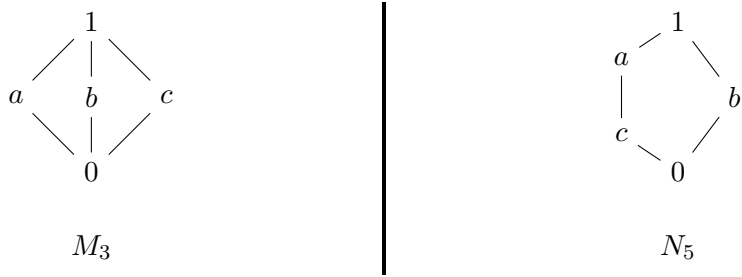
\begin{figure}
\caption{$M_3$ and $N_5$.}
\label{M3N5}
\begin{multicols*}{2}
\begin{tikzpicture}
    \node(v1) at (0,0) {$0$};
    \node(v2) at (-1,1) {$a$};
    \node(v3) at (0,1) {$b$};
    \node(v4) at (1,1) {$c$};
    \node(v5) at (0,2) {$1$};
    \node(v6) at (0,-1) {$M_3$};
    \draw(v1) to (v2);
    \draw(v1) to (v3);
    \draw(v1) to (v4);
    \draw(v2) to (v5);
    \draw(v3) to (v5);
    \draw(v4) to (v5);
\end{tikzpicture}

\columnbreak

\begin{tikzpicture}
    \node(v1) at (0,0) {$0$};
    \node(v2) at (.75,1) {$b$};
    \node(v3) at (-.75,0.5) {$c$};
    \node(v4) at (-.75,1.5) {$a$};
    \node(v5) at (0,2) {$1$};
    \node(v6) at (0, -1) {$N_5$};
    \draw(v1) to (v2);
    \draw(v1) to (v3);
    \draw(v3) to (v4);
    \draw(v2) to (v5);
    \draw(v4) to (v5);
\end{tikzpicture}
\end{multicols*}
\end{figure}

The next lemma gives a precise description of when a lattice is distributive.

\begin{lemma}\cite[Chapter 2, Theorem 1]{latticebook}
A lattice is distributive if and only if it does not have a sub-lattice isomorphic to $M_3$ nor $N_5$.
\end{lemma}

We are ready to state the main result of the section.

\begin{thm}
The set $L=\{ \Gamma_{(d_1,d_2)}(X(a|b))\}_{(d_1,d_2) \geq (0,0)}$ where $\leq$ is defined by inclusions of varieties is a distributive lattice.
\end{thm}

\begin{proof}
Figure \ref{lattice} lists each lattice for each case except the trivial case $(a|b)=(\bar{2}|\bar{3})$. Furthermore, $(L, \leq)$ is distributive because neither $M_3$ nor $N_5$ is isomorphic to a sub-lattice of $L$.
\end{proof}

\begin{figure}\label{lattices}
\caption{The table contains the lattices $(\{ \Gamma_{(d_1,d_2)}(X(a|b))\}_{(d_1,d_2) \geq (0,0)}, \leq )$ in $\IF$ for each possible case except the trivial case when $(a|b)=(\bar{2}|\bar{3})$.}
\label{lattice}
\begin{table}[H]
\scalebox{.7} {\begin{tabular}{|c||c||c|}
\hline
\begin{tikzpicture} 
\node(v1) at (0,0) {$X(a \mid b)$}; 
\node(v2) at (0,2){$X(\bar{2} \mid \bar{3})$}; 
\draw (v1) to (v2); \end{tikzpicture} 
&
\begin{tikzpicture}
\node(v1) at (0,0) {$X(a \mid b)$};
\node(v2) at (0,2) {$\Gamma_{(0,1)}(X(a \mid b))$};
\node(v3) at (0,4) {$X(\bar{2} \mid \bar{3})$};
\draw (v1) to (v2);
\draw (v2) to (v3);
\end{tikzpicture}
&
\begin{tikzpicture}
\node(v1) at (0,0) {$X(a \mid b)$};
\node(v2) at (0,2) {$\Gamma_{(1,0)}(X(a \mid b))$};
\node(v3) at (0,4) {$X(\bar{2} \mid \bar{3})$};
\draw (v1) to (v2);
\draw (v2) to (v3);
\end{tikzpicture}
\\ \hline
If&
If &
If
\\
$a \in \{\bar{2}\}$
&
$a \in \{\bar{3}\}$
&
$b \in \{\bar{2}\}$
\\ 
or
&
and
&   
and
\\ 
$(a \mid b) = (\bar{3} \mid \bar{2})$.
&
$b \not\in \{\bar{2}\}.$
&
$a \not\in \{\bar{3}\}.$
\\ \hline \hline
\begin{tikzpicture}
\node(v1) at (0,0) {$X(a \mid b)$};
\node(v2) at (0,2) {$\Gamma_{(0,1)}(X(a \mid b))$};
\node(v3) at (0,4) {$\Gamma_{(1,1)}(X(a \mid b))$};
\node(v4) at (0,6) {$X(\bar{2} \mid \bar{3})$};
\draw (v1) to (v2);
\draw (v2) to (v3);
\draw (v3) to (v4);
\end{tikzpicture}
&
\begin{tikzpicture}
\node(v1) at (0,0) {$X(a \mid b)$};
\node(v2) at (1.5,2) {$\Gamma_{(1,0)}(X(a \mid b))$};
\node(v3) at (-1.5,2) {$\Gamma_{(0,1)}(X(a \mid b))$};
\node(v4) at (0,4) {$X(\bar{2} \mid \bar{3})$};

\draw (v1) to (v2);
\draw (v1) to (v3);
\draw (v2) to (v4);
\draw (v3) to (v4);
\end{tikzpicture}
&
\begin{tikzpicture}


\node(v1) at (0,0) {$X(a \mid b)$};
\node(v2) at (1.5,2) {$\Gamma_{(1,0)}(X(a \mid b))$};
\node(v3) at (-1.5,2) {$\Gamma_{(0,1)}(X(a \mid b))$};
\node(v4) at (0,4) {$\Gamma_{(1,1)}(X(a \mid b))$};
\node(v5) at (0,6) {$X(\bar{2} \mid \bar{3})$};

\draw (v1) to (v2);
\draw (v1) to (v3);
\draw (v2) to (v4);
\draw (v3) to (v4);
\draw (v4) to (v5);

\end{tikzpicture}\\ \hline
If &
If &
If
\\
$a \not\in \{\bar{2},\bar{3}\}$
&

$b \in \{\bar{3}\}$
& 

$a \not\in \{\bar{2},\bar{3}\}$
\\ 

and
&

and
&

and
\\ 

$b \not\in \{\bar{2},\bar{3}\}$
&

$a \not\in \{\bar{2}\}.$
&

$b \not\in \{\bar{2},\bar{3}\}$
\\ 

and
&


&

and
\\ 

$a>b.$
&


&

$a<b.$\\ 
\hline

\end{tabular}}
\end{table}

\end{figure}

\section{Combinatorial Property $\mathcal{O}$} \label{sec:combinatorial}
We begin by recalling the definitions of the combinatorial quantum Bruhat graph and Combinatorial Conjecuture $\mathcal{O}.$ Let $X$ be a Fano variety. Let $\mathcal{B}:=\{\alpha_i\}_{i \in I}$ denote a basis of the cohomology ring $H^*(X)$. Denote its first Chern class by \[ c_1(X)=a_1 \mbox{Div}_1+a_2 \mbox{Div}_2+\cdots+a_k \mbox{Div}_k\] where $\mbox{Div}_i \in \mathcal{B}$ is a divisor class for each $1 \leq i \leq k$.

\begin{defn} \label{CQBG}The combinatorial quantum Bruhat graph of $X$ is defined as follows. The vertices of this graph are the basis elements $\alpha_i \in H^*(X)$. The edge set is given as follows:
\begin{enumerate}
\item There is an oriented edge $\alpha_i \rightarrow \alpha_j$ if the class $\alpha_j$ appears with positive coefficient in the Chevalley multiplication $h \star \alpha_i$ for some hyperplane class $h$.
\item Let $\alpha_i=[X(i)]$ and $\alpha_j=[X(j)]$. There is an oriented edge $\alpha_i \rightarrow \alpha_j$ 

\begin{enumerate}
\item $X(j)$ is an irreducible component of $\Gamma_d(X(i))$ where $d=(d_1,\cdots,d_k)$,
\item and \[ \dim(X(j))-\dim(X(i))=a_1d_1+a_2d_2+\cdots+a_kd_k-1.\]
\end{enumerate}
\end{enumerate}
\end{defn}

Lemma \ref{lemma: propO} leads us to naturally consider the following combinatorial formulation of Conjecture $\mathcal{O}$.

\begin{defn}
    Combinatorial Property $\mathcal{O}$ holds if the combinatorial quantum Bruhat graph is strongly connected and the Greatest Common Divisor of the cycle lengths is $r=\mbox{GCD}(a_1,a_2,\cdots,a_k).$
\end{defn}

\subsection{Results for $\IF$} \label{CCO}

We begin this section by describing the combinatorial quantum Bruhat graph for $\IF$ which specializes Definition \ref{CQBG}.

\begin{prop}
    Let there be a vertex for each $w \in W^{odd}$. Let $u, v \in W^{odd}$. The combinatorial quantum Bruhat graph $\mathcal{G}_n$ can be created by:
    \begin{enumerate}
        \item There is an arrow $u \rightarrow v$ if $v \leq u$ and $\ell(v) = \ell(u) - 1$;
        
        \item Draw an arrow $u \overset{d}{\rightarrow} v$ if:
        \begin{itemize}
            \item $\Gamma_d(X(u))=X(v_1) \cup \cdots \cup X(v) \cup \cdots \cup X(v_s)$;
            \item $v \nleq v_i$ for any $1 \leq i \leq s$;
            \item $\ell(v)-\ell(u)=2d_1+(2n-1)d_2-1$.
        \end{itemize}
        
    \end{enumerate}
\end{prop}

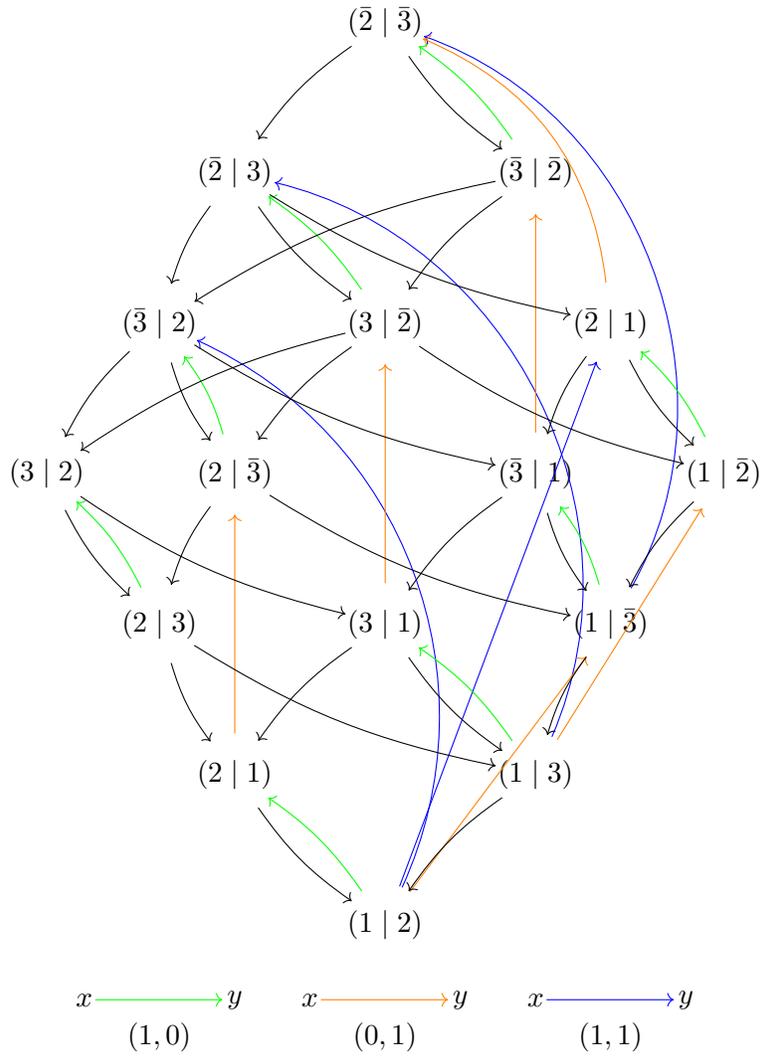
\begin{figure}[H]
\caption{Combinatorial quantum Bruhat graph of $\IF$ for $n = 2$. Notice that the edge joining $(1|2)$ and $(\bar{2}|1)$ is {\bf not} in the moment graph.}
\scalebox{1}{\begin{tikzpicture}

\tikzstyle{vertex} = [circle, minimum size = 2pt, inner sep = 0pt]
\tikzstyle{dedge} = [->, black]
\tikzstyle{uedge} = [->, green]
\tikzstyle{buedge}= [->, orange]
\tikzstyle{bbuedge} = [->, blue]

\node[vertex](v1) at (0,1) {$(1 \mid 2)$};
\node[vertex](v2) at (-2,3) {$(2 \mid 1)$};
\node[vertex](v3) at (2,3) {$(1 \mid 3)$};
\node[vertex](v4) at (-3,5) {$(2 \mid 3)$};
\node[vertex](v5) at (0,5) {$(3 \mid 1)$};
\node[vertex](v6) at (3,5) {$(1 \mid \bar{3})$};
\node[vertex](v7) at (-4.5,7) {$(3 \mid 2)$};
\node[vertex](v8) at (-2,7) {$(2 \mid \bar{3})$};
\node[vertex](v9) at (2,7) {$(\bar{3} \mid 1)$};
\node[vertex](v10) at (4.5,7) {$(1 \mid \bar{2})$};
\node[vertex](v11) at (-3,9) {$(\bar{3} \mid 2)$};
\node[vertex](v12) at (0,9) {$(3 \mid \bar{2})$};
\node[vertex](v13) at (3,9) {$(\bar{2} \mid 1)$};
\node[vertex](v14) at (-2,11) {$(\bar{2} \mid 3)$};
\node[vertex](v15) at (2,11) {$(\bar{3} \mid \bar{2})$};
\node[vertex](v16) at (0,13) {$(\bar{2}\mid \bar{3})$};

\node[vertex](v17) at (-1,0) {$x$};
\node[vertex](v18) at (1,0) {$y$};
\node[vertex](v19) at (-4,0) {$x$};
\node[vertex](v20) at (-2,0) {$y$};
\node[vertex](v21) at (2,0) {$x$};
\node[vertex](v22) at (4,0) {$y$};

\node[vertex](v23) at (-3,-.5) {$(1,0)$};
\draw[uedge](v19) to  (v20);

\node[vertex](v24) at (0,-.5) {$(0,1)$};
\draw[buedge](v17)--(v18);

\node[vertex](v25) at (3,-.5) {$(1,1)$};
\draw[bbuedge](v21)--(v22);

\draw[bbuedge](v1)to[bend right=45](v11);
\draw[bbuedge](v3)to[bend right=50](v14);
\draw[bbuedge](v6)to[bend right=50](v16);
\draw[bbuedge](v1)to (v13);

\draw[buedge](v1)to(v6);
\draw[buedge](v2)to(v8);
\draw[buedge](v3)to(v10);
\draw[buedge](v5)to(v12);
\draw[buedge](v9)to(v15);
\draw[buedge](v13)to[bend right=30](v16);

\draw[uedge](v1)to[bend right=10](v2);
\draw[uedge](v3)to[bend right=10](v5);
\draw[uedge](v4)to[bend right=10](v7);
\draw[uedge](v6)to[bend right=10](v9);
\draw[uedge](v10)to[bend right=10](v13);
\draw[uedge](v8)to[bend right=10](v11);
\draw[uedge](v12)to[bend right=10](v14);
\draw[uedge](v15)to[bend right=10](v16);

\draw[dedge](v2)to[bend right=10](v1);
\draw[dedge](v3)to[bend right=10](v1);
\draw[dedge](v4)to[bend right=10](v2);
\draw[dedge](v4)to[bend right=10](v3);
\draw[dedge](v5)to[bend right=10](v2);
\draw[dedge](v5)to[bend right=10](v3);
\draw[dedge](v6)to[bend right=10](v3);
\draw[dedge](v7)to[bend right=10](v4);
\draw[dedge](v8)to[bend right=10](v4);
\draw[dedge](v8)to[bend right=10](v6);
\draw[dedge](v9)to[bend right=10](v5);
\draw[dedge](v9)to[bend right=10](v6);
\draw[dedge](v10)to[bend right=10](v6);
\draw[dedge](v11)to[bend right=10](v7);
\draw[dedge](v11)to[bend right=10](v8);
\draw[dedge](v11)to[bend right=10](v9);
\draw[dedge](v12)to[bend right=10](v7);
\draw[dedge](v12)to[bend right=10](v8);
\draw[dedge](v12)to[bend right=10](v10);
\draw[dedge](v13)to[bend right=10](v9);
\draw[dedge](v13)to[bend right=10](v10);
\draw[dedge](v14)to[bend right=10](v11);
\draw[dedge](v14)to[bend right=10](v12);
\draw[dedge](v14)to[bend right=10](v13);
\draw[dedge](v15)to[bend right=10](v11);
\draw[dedge](v15)to[bend right=10](v12);
\draw[dedge](v16)to[bend right=10](v14);
\draw[dedge](v16)to[bend right=10](v15);

\draw[dedge](v7)to[bend right=10](v5);

\end{tikzpicture}
}
\end{figure}
\newpage
\begin{remark} \label{redcrv}
    Observe that the edge $(1|2) \overset{(1,1)}{\rightarrow}(\bar{2}|1)$ in the combinatorial quantum Bruhat graph is {\bf not} present in the moment graph. There is no known example of an edge appearing in the (geometric) quantum Bruhat graph that does not appear in the moment graph.
\end{remark}

We are ready to state and prove the main result of this section.

\begin{thm} \label{Omainresult}
    Combinatorial Property $\mathcal{O}$ holds for $\IF$.
\end{thm}

\begin{proof}
We begin the proof by showing that combinatorial quantum Bruhat graph is strongly connected. First, there is a path from $(1|2)$ to $(\bar{2}|\bar{3})$ in the combinatorial quantum Bruhat graph given by \[(1|2) \overset{(1,1)}{\rightarrow} (\bar{2}|1) \overset{(0,1)}{\rightarrow} (\bar{2}|\bar{3}) \] since \[ \Gamma_{(1,1)}(X(1|2))=X(\bar{2}|1), \Gamma_{(0,1)}(X(\bar{2}|1))=X(\bar{2}|\bar{3}),\] \[\ell(\bar{2}|1)-\ell(1|2)=2*1+(2n-1)*1-1, \mbox{ and }\ell(\bar{2}|\bar{3})-\ell(\bar{2}|1)=2*0+(2n-1)*1-1.\] Next, there is clearly a path from $(a|b)$ to $(1|2)$ by decomposing the permutation $(a|b)$ into simple reflections. Finally we claim that there is a path from $(\bar{2}|\bar{3})$ to any other vertex. Indeed, if $(a|b) \in W^{odd} \backslash  \{(\bar{2}|\bar{3})\}$ then there is another point $(c|d) \in W^{odd}$ and edge $(c|d) \rightarrow (a|b)$ such that $\ell(c|d)-\ell(a|b)=1$. Since the $\ell(\bar{2}|\bar{3})$ is maximum in the Bruhat order, we conclude the combinatorial quantum Bruhat graph is strongly connected.

The quantum Bruhat graph $\mathcal{G}_n$ has a cycle of length 2 given by
\[(1|2) \rightarrow (2|1)\rightarrow (1|2).\] There is a cycle of length $2n-1$ given by one of the following two cases. For $n=2$ we have
\[(1|2) \rightarrow (1|\bar{3}) \rightarrow(1|3) \rightarrow (1|2).\]
For $n>2$ we have
\[(1|2) \rightarrow (1|\bar{3}) \rightarrow \cdots  \rightarrow (1|\bar{n}) \rightarrow(1|n) \rightarrow \cdots \rightarrow (1|2).\] The result follows.
\end{proof}

\bibliographystyle{halpha}
\bibliography{bibliography}

\begin{bibdiv}
\begin{biblist}

\bib{pech:quantum}{article}{
      author={{ Cl{\'e}lia Pech}},
       title={{Quantum cohomology of the odd symplectic Grassmannian of
  lines}},
        date={{2013}},
        ISSN={{0021-8693}},
     journal={{J. Algebra}},
      volume={{375}},
       pages={{188\ndash 215}},
}

\bib{Songul:thesis}{article}{
      author={Aslan, Songul},
       title={The combinatorial curve neighborhoods of affine flag manifold in
  type ${A}_{n-1}^{(1)}$},
        date={2018},
     journal={Ph.D. Thesis, Virginia Tech},
}

\bib{BCMP:qkfin}{article}{
      author={Buch, Anders},
      author={Chaput, Pierre-Emmanuel},
      author={Mihalcea, Leonardo~C.},
      author={Perrin, Nicolas},
       title={Finiteness of cominuscule quantum {K}-theory},
        date={2013},
     journal={Annales Sci. de L'\'Ecole Normale Sup\'erieure},
      number={46},
       pages={477\ndash 494},
}

\bib{BFSS}{article}{
      author={Bones, Lela},
      author={Fowler, Garrett},
      author={Schneider, Lisa},
      author={Shifler, Ryan~M.},
       title={Conjecture $\mathcal{O}$ holds for some horospherical varieties
  of {P}icard rank 1},
        date={2020},
      volume={13},
      number={4},
       pages={551\ndash 558},
}

\bib{buch.m:nbhds}{article}{
      author={Buch, Anders},
      author={Mihalcea, Leonardo},
       title={Curve neighborhoods of {S}chubert varieties},
        date={2015},
        ISSN={0022-040X},
     journal={J. Differential Geom.},
      volume={99},
      number={2},
       pages={255\ndash 283},
         url={http://projecteuclid.org/euclid.jdg/1421415563},
}

\bib{ChLi}{article}{
      author={Cheong, Daewoong},
      author={Li, Changzheng},
       title={On the {C}onjecture $\mathcal{O}$ of {GGI} for {$G/P$}},
        date={2017},
        ISSN={0001-8708},
     journal={Adv. Math.},
      volume={306},
       pages={704\ndash 721},
         url={http://dx.doi.org/10.1016/j.aim.2016.10.033},
}

\bib{GGI}{article}{
      author={Galkin, Sergey},
      author={Golyshev, Vasily},
      author={Iritani, Hiroshi},
       title={Gamma classes and quantum cohomology of {F}ano manifolds: gamma
  conjectures},
        date={2016},
        ISSN={0012-7094},
     journal={Duke Math. J.},
      volume={165},
      number={11},
       pages={2005\ndash 2077},
         url={http://dx.doi.org/10.1215/00127094-3476593},
}

\bib{latticebook}{book}{
      author={Gr\"atzer, George},
       title={General lattice theory, second edition},
   publisher={Birkh\"auser Boston, Inc., Boston, MA},
        date={1998},
        ISBN={0-12-295750-4; 3-7643-5239-6; 3-7643-6996-5},
}

\bib{gelfand.zelevinsky}{article}{
      author={Gel{\cprime}fand, I.~M.},
      author={Zelevinski{\u{\i}}, A.~V.},
       title={Models of representations of classical groups and their hidden
  symmetries},
        date={1984},
        ISSN={0374-1990},
     journal={Funktsional. Anal. i Prilozhen.},
      volume={18},
      number={3},
       pages={14\ndash 31},
}

\bib{delpezzo}{article}{
      author={Hu, Jianxun},
      author={Ke, Huazhong},
      author={Li, Changzheng},
      author={Yang, Tuo},
       title={Gamma conjecture {I} for del {P}ezzo surfaces},
        date={2021},
        ISSN={0001-8708},
     journal={Advances in Mathematics},
      volume={386},
       pages={107797},
  url={https://www.sciencedirect.com/science/article/pii/S000187082100236X},
}

\bib{projinter}{article}{
      author={Ke, Huazhong},
       title={On {C}onjecture $\mathcal{O}$ for projective complete
  intersections},
        date={2023},
     journal={International Mathematics Research Notices},
      volume={2024},
       pages={3947–3974},
}

\bib{kumar:kacmoody}{book}{
      author={Kumar, Shrawan},
       title={Kac-{M}oody groups, their flag varieties and representation
  theory},
      series={Progress in Mathematics},
   publisher={Birkh\"auser Boston, Inc., Boston, MA},
        date={2002},
      volume={204},
}

\bib{LMS}{article}{
      author={Li, Changzheng},
      author={Mihalcea, Leonardo~C.},
      author={Shifler, Ryan~M.},
       title={Conjecture $\mathcal{O}$ holds for the odd symplectic
  {G}rassmannian},
        date={2019},
     journal={Bulletin of the London Mathematical Society},
      volume={51},
      number={4},
       pages={705\ndash 714},
  url={https://londmathsoc.onlinelibrary.wiley.com/doi/abs/10.1112/blms.12268},
}

\bib{mihai:odd}{article}{
      author={Mihai, Ion~Alexandru},
       title={Odd symplectic flag manifolds},
        date={2007},
        ISSN={1083-4362},
     journal={Transform. Groups},
      volume={12},
      number={3},
       pages={573\ndash 599},
         url={http://dx.doi.org/10.1007/s00031-006-0053-0},
}

\bib{mare.mihalcea}{article}{
      author={Mare, Augustin-Liviu},
      author={Mihalcea, Leonardo~C.},
       title={An affine quantum cohomology ring for flag manifolds and the
  periodic toda lattice},
        date={2018},
     journal={Proceedings of the London Mathematical Society},
      volume={116},
      number={1},
       pages={135\ndash 181},
  url={https://londmathsoc.onlinelibrary.wiley.com/doi/abs/10.1112/plms.12077},
}

\bib{mihalcea.shifler:qhodd}{article}{
      author={Mihalcea, Leonardo~C.},
      author={Shifler, Ryan~M.},
       title={Equivariant quantum cohomology of the odd symplectic
  {G}rassmannian},
        date={2019},
        ISSN={0025-5874},
     journal={Math. Z.},
      volume={291},
      number={3-4},
       pages={1569\ndash 1603},
         url={https://doi.org/10.1007/s00209-018-2120-3},
}

\bib{proctor:oddsymgrps}{article}{
      author={Proctor, Robert~A.},
       title={Odd symplectic groups},
        date={1988},
        ISSN={0020-9910},
     journal={Invent. Math.},
      volume={92},
      number={2},
       pages={307\ndash 332},
         url={http://dx.doi.org/10.1007/BF01404455},
}

\bib{PechShif:CNBDS}{article}{
      author={Pech, Cl\'elia},
      author={Shifler, Ryan~M.},
       title={Curve neighborhoods of {S}chubert varieites in the odd symplectic
  {G}rassmannian},
        date={2024},
     journal={Transformation Groups},
      volume={19},
      number={1},
       pages={361\ndash 408},
}

\bib{Maya}{article}{
      author={Shifler, Ryan~M.},
       title={Minimum quantum degrees with {M}aya diagrams},
        date={2025},
     journal={Annals of Combinatorics, to appear},
}

\bib{ShiflerWithrow}{article}{
      author={Shifler, Ryan~M.},
      author={Withrow, Camron},
       title={Minimum quantum degrees for isotropic {G}rassmannians in {T}ypes
  {B} and {C}},
        date={2022},
     journal={Annals of Combinatorics},
      volume={26},
      number={2},
       pages={453\ndash 480},
}

\end{biblist}
\end{bibdiv}

\end{document}